\documentclass[12pt]{amsart}
\usepackage{amsmath}
\usepackage{amsfonts}
\usepackage{amsthm}
\usepackage{amssymb}
\usepackage{amscd}
\usepackage[all]{xy}
\usepackage{enumerate}

\keywords{Extension class of a vector bundle, torsion freeness, Castelnuovo's free pencil trick,
infinitesimal Torelli problem, projective hypersurface, meromorphic forms.} 

\subjclass[2010]{14C34, 14D07, 14J10, 14J40, 14J70}

\theoremstyle{plain}
\newtheorem{thm}{Theorem}[subsection]

\newtheorem{prop}[thm]{Proposition}

\newtheorem{cor}[thm]{Corollary}

\newtheorem{lem}[thm]{Lemma}

\theoremstyle{definition}
\newtheorem{defn}[thm]{Definition}

\newtheorem{rmk}[thm]{Remark}

\newcommand{\sB}{\mathcal{B}}

\newcommand{\sE}{\mathcal{E}}
\newcommand{\sF}{\mathcal{F}}
\newcommand{\sG}{\mathcal{G}}

\newcommand{\sI}{\mathcal{I}}
\newcommand{\sJ}{\mathcal{J}}

\newcommand{\sL}{\mathcal{L}}

\newcommand{\sM}{\mathcal{M}}
\newcommand{\sO}{\mathcal{O}}

\newcommand{\sV}{\mathcal{V}}

\newcommand{\mC}{\mathbb{C}}

\newcommand{\mP}{\mathbb{P}}

\newcommand{\Ima}{\mathrm{Im}\,}

\newcommand{\rank}{\mathrm{rank}\,}

\numberwithin{equation}{section}

\newcommand{\beba}  {\begin{equation}\begin{array}{rcl}}

\newcommand{\eaee}  {\end{array}\end{equation}}

\let\oldtocsection=\tocsection

\let\oldtocsubsection=\tocsubsection

\let\oldtocsubsubsection=\tocsubsubsection

\renewcommand{\tocsection}[2]{\hspace{0em}\oldtocsection{#1}{#2}}
\renewcommand{\tocsubsection}[2]{\hspace{1em}\oldtocsubsection{#1}{#2}}
\renewcommand{\tocsubsubsection}[2]{\hspace{2em}\oldtocsubsubsection{#1}{#2}}

\title{Generalized Adjoint Forms on Algebraic Varieties}

\author{Luca Rizzi}
\address{D.I.M.I. \\
the University of Udine\\
Udine, 33100 Italy\\
\texttt{rizzi.luca@spes.uniud.it}}

\author{Francesco Zucconi}

\address{D.I.M.I. \\
the University of Udine\\
Udine, 33100 Italy\\
\texttt{Francesco.Zucconi@dimi.uniud.it}}

\begin{document}

\markboth{Rizzi and Zucconi}{Generalized Adjoint Forms on Algebraic Varieties}

\begin{abstract} We prove a full  generalization of the Castelnuovo's free pencil trick. We show its analogies with \cite[Theorem 2.1.7]{RZ1}; see also \cite[Theorem 1.5.1]{PZ}. Moreover we find a new formulation 
of the Griffiths's infinitesimal Torelli Theorem for smooth projective hypersurfaces using meromorphic $1$-forms.
\end{abstract}

\maketitle
\tableofcontents

\section{Introduction}
Let $X$ be an $m$-dimensional smooth projective variety and $\sF$ be a rank $n$ locally free sheaf over it.  A way to study $\sF$ is to study its extensions $0\to\sL\to \sE\to\sF\to 0$ which,  up to isomorphism,  are parametrized by $\text{Ext}^1(\sF,\sL)$.  In \cite{CP}, \cite{PZ}, \cite{RZ1}, \cite{Ra}, \cite{PR}, \cite{CNP}, \cite{victor1}, \cite{victor2} and \cite{BGN} the adjoint forms associated to $\xi\in {\rm{Ext}}^{1}(\sO_X, \sF)$ are deeply studied and many applications are given. Let us recall the notion of adjoint form in the case $\sL=\sO_X$.

Given  $\xi\in {\rm{Ext}}^{1}(\sO_X, \sF)$, take an $(n+1)$-dimensional subspace $W$ of the kernel of the cup-product homomorphism $\partial_{\xi}\colon H^0(X,\sF)\to H^1(X,\sO_X)$. Denote by $\lambda^{i}W$ the image of $\bigwedge^iW$ through the natural homomorphism $\lambda^i\colon\bigwedge^i H^0(X,\sF)\to  H^0(X,\bigwedge^{i}\sF)$.
If $\sB:=\langle\eta_{1},\ldots,\eta_{n+1}\rangle$ is a basis of $W$ and $s_1,\ldots, s_{n+1}\in H^0(X,\sE)$ are liftings of  $\eta_{1},\ldots,\eta_{n+1}$ respectively, then the map $\Lambda^{n+1}\colon\bigwedge^{n+1} H^0(X,\sE)\to  H^0(X,\bigwedge^{n+1}\sE)$ gives
 the top form $ \Omega:=\Lambda^{n+1}(s_1\wedge s_2 \wedge\ldots\wedge s_{n+1})\in H^0(X,\det\sE)$. The section $\Omega$ corresponds to a top form $\omega_{\xi,W,{\widehat\sB}}\in H^0(X,\det\sF )$ via the isomorphism $\det\sF \simeq\det\sE$, where ${\widehat\sB}=\langle s_1,\ldots, s_{n+1}\rangle$; the form  $\omega_{\xi,W,{\widehat\sB}}$ is called {\it{an adjoint form of $W$ and $\xi$.}}  To the basis $\sB$ there are also naturally associated $n+1$ elements $\omega_i:=\lambda^{n}(\eta_1\wedge\ldots\wedge\eta_{i-1}\wedge{\widehat{\eta_i}}\wedge\eta_{i+1}\wedge\ldots\wedge\eta_{n+1})$, $i=1,\ldots,n+1$ obtained by the basis 
 $\langle \eta_1\wedge\ldots\wedge\eta_{i-1}\wedge{\widehat{\eta_i}}\wedge\eta_{i+1}\wedge\ldots\wedge\eta_{n+1}\rangle_{i=1}^{n+1}$ of $\bigwedge^n W$.  Note that if we change the liftings $s_1,\ldots, s_{n+1}\in H^0(X,\sE)$ with other liftings ${\widetilde s_1},\ldots, {\widetilde s_{n+1}}$ then  $\omega_{\xi,W,\widehat \sB}$ is a linear combination of $\omega_{\xi,W,\widetilde \sB}$ and $\omega_1$,\ldots, $\omega_{n+1}$. The natural problem of this theory is to characterize the condition $\omega_{\xi,W,\widehat \sB}\in \lambda^nW$ in terms of the fixed divisor $D_W$ of $|\lambda^nW|\subset \mathbb P H^0(X,\det\sF)$ and of the base locus $Z_W$ of the moving part $M_W\in \mathbb P H^0(X,\det\sF\otimes_{\sO_X}\sO_{X}(-D_W))$, where $|\lambda^nW|=D_W+|M_W|$.

 In this paper we consider the general case where $\sL$ is an invertible sheaf not necessarily equal to $\sO_X$. In this case $ \det\sE=\sL\otimes \det\sF$ and liftings  $s_1,\ldots, s_{n+1}\in H^0(X,\sE)$ of $\eta_{1},\ldots,\eta_{n+1}\in H^0(X,\sF)$ determine 
 $\Omega:=\Lambda^{n+1}(s_1\wedge s_2 \wedge\ldots\wedge s_{n+1})\in H^0(X,\det\sE)$ which is now called a generalized adjoint form. We define as before $\omega_i:=\lambda^{n}(\eta_1\wedge\ldots\wedge\eta_{i-1}\wedge{\widehat{\eta_i}}\wedge\eta_{i+1}\wedge\ldots\wedge\eta_{n+1})$, $i=1,\ldots, n+1$ and we characterize the case where
 $\Omega$ belongs to the image of $H^0(X,\sL)\otimes  \lambda^nW\to  H^0(X,\det\sE)$ by the natural tensor product map. The game is more complicated than in the above mentioned papers because the linear system $|\lambda^nW|$ is inside $\mathbb P H^0(X,\det\sF)$ and we have to relate the fixed divisor $D_W$ of $|\lambda^nW|$ and the base locus $Z_W$ of the moving part $M_W\in \mathbb P H^0(X,\det\sF\otimes_{\sO_X}\sO_{X}(-D_W))$ to forms which are not anymore inside $H^0(X,\det\sF)$. Nevertheless the result is analogue to the one of \cite[Theorem 1.5.1]{PZ} and \cite[2.1.7]{RZ1}:
 \medskip
 
\noindent
 {\bf{Theorem [A]}} 
 {\it{ 
 Let $X$ be an $m$-dimensional complex compact smooth variety. 
 Let $\sF$ be a rank $n$ locally free sheaf on $X$ and $\sL$ an invertible sheaf. 
 Consider an extension $0\to\sL\to\mathcal{E}\to\mathcal{F}\to0$ corresponding to $\xi\in \text{Ext}^1(\sF,\sL)$. 
 Let $W=\langle \eta_1,\ldots,\eta_{n+1}\rangle$ be an $n+1$-dimensional sublinear system of $\ker({\partial_\xi})\subset H^0(X,\mathcal{F})$. 
 Let $\Omega\in H^0(X,\det\sE)$ be  an adjoint form associated to $W$ as above.
It holds that if $\Omega \in \Ima(H^0(X,\sL)\otimes \lambda^nW\to H^0(X,\det\sE))$ 
then $\xi\in\ker(H^1(X,\sF^\vee\otimes\sL)\to H^1(X,\sF^\vee\otimes\sL(D_W)))$.
}}
\medskip

Theorem [A], called Adjoint Theorem, can be thought as a general version of the well-known Castelnuovo's free pencil trick; c.f. see Theorem \ref{rango1}. 

We have also a viceversa of the Adjoint Theorem; see: Theorem \ref{inversoaggiunta}:

 \medskip

\noindent
 {\bf{Theorem [B]}} 
 {\it{ Under the same hypothesis of Theorem [A], assume also that $H^0(X,\sL)\cong H^0(X,\sL(D_W))$.
  It holds that  if $\xi\in\ker(H^1(X,\sF^\vee\otimes\sL)\to H^1(X,\sF^\vee\otimes\sL(D_W)))$, then $\Omega \in \Ima(H^0(X,\sL)\otimes \lambda^nW\to H^0(X,\det\sE))$.
 }}
  \medskip

 In particular in the case $D_W=0$ Theorem [B] is a full characterization of the condition  $\Omega \in \Ima(H^0(X,\sL)\otimes \lambda^nW\to H^0(X,\det\sE))$.

 Now by the Adjoint Theorem and by Theorem [B] we can study extension classes of sheaves via adjoint forms. Indeed even if $\sF$ has no global sections we can always take the tensor product with a sufficiently ample linear system $\sM$ such that $\sF\otimes\sM$ has enough global sections in order to apply the theory of adjoint forms. By applying the above idea to the case where $n>2$, $X\subset\mathbb P^n$ is an hypersurface of degree $d>3$ and $\sF:=\Omega^{1}_{X}\otimes_{\sO_X}\sO_X(2)$ we have a reformulation of the infinitesimal Torelli Theorem for $X$ in the setting of generalized adjoint theory. In this paper we will not recall the theory concerning infinitesimal Torelli Theorems, for which a reference is \cite{Vo2}, in any case a quick introduction to this topic is also given in \cite{RZ1}. Here we point out only that given a degree $d$ form $F\in \mC[\xi_0,\ldots,\xi_n]$ the Jacobian ideal of $F$ is the ideal $\sJ$ generated by the partial derivatives $\frac{\partial{F}}{\partial{\xi_i}}$ for $i=0,\ldots,n$ and by \cite{Gri1}[Theorem 9.8], any infinitesimal deformation $\xi\in H^1(X,\Theta_X)$, where $X=(F=0)$ and $\Theta_X$ is the sheaf of tangent vectors on $X$, is given by a class $[R]$  in the quotient $\mC[\xi_0,\ldots,\xi_n]/\sJ$ where $R$ is a homogeneous form of degree $d$.
  \medskip

\noindent
{\bf{Theorem [C]}} 
{\it{
For a smooth hypersurface $X$ of degree $d$ in $\mP^n$ with $n\geq3$ and $d>3$ the following are equivalent:
\begin{itemize}
	\item[\it{i})] the differential of the period map  is zero on the infinitesimal deformation $$[R]\in( \mC[\xi_0,\ldots,\xi_n]/\sJ)_{d}\simeq H^1(X,\Theta_X)$$
	\item[\it{ii})] $R$ is an element of the Jacobian ideal $\sJ$
	\item[\it{iii})] $\Omega \in \Ima(H^0(X,\sO_X(2))\otimes \lambda^nW \to H^0(X,\sO_X(n+d-1)))$ for the generic generalized adjoint $\Omega$
	\item[\it{iv})] The generic generalized adjoint $\Omega$ lies in $\sJ$.
\end{itemize}
}}
\medskip

Note that Theorem [C] has a different flavor with respect to the analogue \cite[Theorem 9.8]{Gri1} since we essentially use meromorphic $1$-forms over $X$; see Proposition \ref{dachiamare}. Finally we want to mention that in a forthcoming paper \cite{RZ2} we show how to recover also the Green's infinitesimal Torelli Theorem for a sufficiently ample divisor of a smooth variety in terms of generalized adjoint theory.

\section{The theory of generalized adjoint forms}
\subsection{Definition of generalized adjoint form}
Let $X$ be a smooth compact complex variety of dimension $m$ and let $\sF$ and $\sL$ be two locally free sheaves on $X$ of rank $n$ and $1$ respectively. Consider the exact sequence of locally free sheaves
\begin{equation}
\label{sequenza}
0\to\sL\to \sE\to\sF\to 0
\end{equation} associated to an element $\xi\in \text{Ext}^1(\sF,\sL)\cong H^1(X,\sF^\vee\otimes\sL)$. Recall that the invertible sheaf $\det\sF:=\bigwedge^n\sF$ fits into the exact sequence
\begin{equation}
\label{wedge}
0\to\bigwedge^{n-1}\sF\otimes\sL\to \bigwedge^n\sE\to \det\sF\to 0, 
\end{equation} which still corresponds to $\xi$ under the isomorphism $\text{Ext}^1(\sF,\sL)\cong\text{Ext}^1(\det\sF,\bigwedge^{n-1}\sF\otimes\sL)\cong H^1(X,\sF^\vee\otimes\sL)$. Furthermore $\det\sF$ satisfies
\begin{equation}
\label{isom}
\det\sF\otimes\sL\cong\det\sE.
\end{equation} Let $\partial_\xi \colon H^0(X,\sF)\to H^1(X,\sL)$ be the connecting homomorphism related to (\ref{sequenza}), and let $W\subset \ker(\partial_\xi)$ be a vector subspace of dimension $n+1$. Choose a basis $\mathcal{B}:=\{\eta_1,\ldots,\eta_{n+1}\}$ of $W$. By definition we can take liftings $s_1,\ldots,s_{n+1}\in H^0(X,\sE)$ of the sections $\eta_1,\ldots,\eta_{n+1}$. If we consider the natural map
\begin{equation*}
\Lambda^n\colon \bigwedge^{n}H^0(X,\sE)\to H^0(X,\bigwedge^n\sE)
\end{equation*} we can define the sections
\begin{equation}
\label{Omegai}
\Omega_i:=\Lambda^n(s_1\wedge\ldots\wedge\hat{s_i}\wedge\ldots\wedge s_{n+1})
\end{equation} for $i=1,\ldots,n+1$. Denote by $\omega_i$, for $i=1,\ldots,n+1$, the corresponding sections in $H^0(X,\det\sF)$. Obviously we have that $\omega_i=\lambda^n(\eta_1\wedge\ldots\wedge\hat{\eta_i}\wedge\ldots\wedge \eta_{n+1})$, where $\lambda^n$ is the natural morphism
\begin{equation*}
\lambda^n\colon \bigwedge^{n}H^0(X,\sF)\to H^0(X,\det\sF).
\end{equation*} The vector subspace of $H^0(X,\det\sF)$ generated by $\omega_1,\ldots,\omega_{n+1}$ is denoted by $\lambda^nW$.

\begin{defn}
If $\lambda^nW$ is nontrivial, it induces a sublinear system $|\lambda^n W|\subset \mP(H^0(X,\det\sF))$ that we will call \emph{adjoint sublinear system}. We call $D_W$ its fixed divisor and $Z_W$ the base locus of its moving part $|M_W|\subset\mP(H^0(X,\det\mathcal{F}(-D_W)))$.
\end{defn}

\begin{defn}
The section $\Omega\in H^0(X,\det\sE)$ corresponding to $s_1\wedge\ldots\wedge s_{n+1}$ via 
\begin{equation}
\Lambda^{n+1}\colon \bigwedge^{n+1}H^0(X,\sE)\to H^0(X,\det\sE)
\end{equation} is called generalized adjoint form.
\end{defn} 
\begin{rmk}
\label{zeri}
It is easy to see by local computation that this section is in the image of the natural injection $\det\sE(-D_W)\otimes\sI_{Z_W}\to \det\sE$.
\end{rmk}

We want to study the condition
\begin{equation}
\label{aggiuntazero1}
\Omega \in \Ima(H^0(X,\sL)\otimes \left\langle \Omega_i\right\rangle\to H^0(X,\det\sE))
\end{equation} or, equivalently,
\begin{equation}
\label{aggiuntazero2}
\Omega \in \Ima(H^0(X,\sL)\otimes \lambda^nW \to H^0(X,\det\sE)).
\end{equation} The first map is given by the wedge product, the second one by (\ref{isom}). Note that if $H^0(X,\sL)=0$ this condition is equivalent to $\Omega=0$.

\begin{rmk}
The choice of the liftings is not relevant for this purpose. Take different liftings $s_1',\ldots,s_{n+1}'\in H^0(X,\sE)$ of $\eta_1,\ldots,\eta_{n+1}$ and call $\Omega_i'\in H^0(X,\bigwedge^n\sE)$ and $\Omega'\in H^0(X,\det\sE)$ the corresponding sections constructed as above. Obviously 
\begin{equation}
\Ima(H^0(X,\sL)\otimes \left\langle \Omega_i\right\rangle\to H^0(X,\det\sE))=\Ima(H^0(X,\sL)\otimes \left\langle \Omega_i'\right\rangle\to H^0(X,\det\sE)),
\end{equation} since they are both equal to $\Ima(H^0(X,\sL)\otimes \lambda^nW \to H^0(X,\det\sE))$. It is also easy to see that $\Omega\in\Ima(H^0(X,\sL)\otimes \left\langle \Omega_i\right\rangle\to H^0(X,\det\sE))$ iff $\Omega'\in\Ima(H^0(X,\sL)\otimes \left\langle \Omega_i'\right\rangle\to H^0(X,\det\sE))$.
\end{rmk}

\begin{rmk}
Consider another basis $\mathcal{B}':=\{\eta_1',\ldots,\eta_{n+1}'\}$ of $W$ and let $A$ be the matrix of the basis change. 
The sections $s_1',\ldots,s_{n+1}'$ obtained from $s_1,\ldots,s_{n+1}$ through the matrix $A$ are liftings of $\eta_1',\ldots,\eta_{n+1}'$. The section $\Omega':=\Lambda^{n+1}(s_1'\wedge\ldots\wedge s_{n+1}')$ satisfies $\Omega'=\det A\cdot\Omega$. Moreover $\Omega\in\Ima(H^0(X,\sL)\otimes \left\langle \Omega_i\right\rangle\to H^0(X,\det\sE))$ iff $\Omega'\in\Ima(H^0(X,\sL)\otimes \left\langle \Omega_i'\right\rangle\to H^0(X,\det\sE))$.
\end{rmk}

\begin{lem}
\label{sollevamenti}
If $\Omega \in \Ima(H^0(X,\sL)\otimes \left\langle \Omega_i\right\rangle\to H^0(X,\det\sE))$, then we can find liftings 
$\tilde{s_i}\in H^0(X,\mathcal{E})$, $i=1,\ldots,n+1$, such that $\tilde{\Omega}:=\Lambda^{n+1}(\tilde{s_1}\wedge\ldots\wedge \tilde{s}_{n+1})=0$.
\end{lem}
\begin{proof}
By hypothesis there exist $\sigma_i\in H^0(X,\sL)$ such that
\begin{equation}
\label{ipotesi}
\Omega=\sum^{n+1}_{i=1} \sigma_i\wedge\Omega_i
\end{equation}
We can define new liftings for the element $\eta_i$:
\begin{equation*}
\tilde{s_i}:=s_i+(-1)^{n-i}\sigma_i.
\end{equation*} Now, since
\begin{equation}
\tilde{s_1}\wedge \ldots\wedge \tilde{s}_{n+1}=s_1\wedge \ldots\wedge s_{n+1}-\sum^{n+1}_{i=1}  s_1\wedge\ldots\wedge\hat{s_i}\wedge\ldots\wedge s_{n+1}\wedge\sigma_i,
\end{equation} we immediately deduce $\tilde{\Omega}=0$.
\end{proof}

From the natural map
\begin{equation*}
\sF^\vee\otimes\sL\to\sF^\vee\otimes\sL(D_W)
\end{equation*} we have a homomorphism
\begin{equation*}
H^1(X,\sF^\vee\otimes\sL)\stackrel{\rho}{\rightarrow} H^1(X,\sF^\vee\otimes\sL(D_W));
\end{equation*} we call $\xi_{D_W}=\rho(\xi)$. 

\subsection{Castelnuovo's free pencil trick}
Consider the case where both $\sL$ and $\sF$ are of rank one, while $X$ has dimension $m$. In this case $W=\left\langle \eta_1,\eta_2\right\rangle\subset H^0(X,\sF)$ has dimension two; as usual we choose liftings $s_1,s_2\in H^0(X,\sE)$ of $\eta_1,\eta_2$. Note also that $\omega_1=\eta_2$ and $\omega_2=\eta_1$, in particular $W=\lambda^1W$ so $D_W$ is the fixed part of $W$ and $Z_W$ is the base locus of its moving part. Call $\tilde{\eta}_i	\in H^0(X,\sF(-D_W))$ the sections corresponding to the $\eta_i$'s via $H^0(X,\sF(-D_W))\to H^0(X,\sF)$.
The following lemma is well known and it is the core of the Castelnuovo base point free pencil trick.
\begin{lem}
We have an exact sequence
\begin{equation}
\label{castelnuovo}
0\to\mathcal{F}^\vee(D_W)\stackrel{i}{\rightarrow}\mathcal{O}_X\oplus\mathcal{O}_X\stackrel{\nu}{\rightarrow}\mathcal{F}(-D_W)\otimes \sI_{Z_W}\to 0
\end{equation}
where the morphism $i$ is given by contraction with $-\tilde{\eta}_1$ and $\tilde{\eta}_2$, while $\nu$ is given by evaluation with $\tilde{\eta}_2$ on the first component and $\tilde{\eta}_1$ on the second one.
\end{lem}
%\begin{proof}
%Everything is trivial, here we only prove that $\ker{\nu}\subset\Ima{i}$. As usual this is a computation on the stalks. Take $(f,g)\in\sO_X(U)\oplus\sO_X(U)$, where $U$ is an open set of $X$. Up to shrinking $U$, there exists a local section $\sigma$ which trivializes $\sF$, and we can write $\eta_i=a_i\cdot\sigma$, with $a_i\in \sO_X(U)$. If $(f,g)\in\ker\nu$, we have the relation
%\begin{equation}
%\label{relazione}
%f\cdot\frac{a_2}{d}+g\cdot\frac{a_1}{d}=0,
%\end{equation} where $d$ is a local expression of $D_W$. Take a point $P\in U$. If $P$ is not in $\text{Supp}(Z_W)$, then at least one between $\frac{a_2}{d}$ and $\frac{a_1}{d}$ is invertible in a neighborhood of $P$. For example let the germ of $\frac{a_1}{d}$ be nonzero in $P$. Then by (\ref{relazione}), the local section of $\mathcal{F}^\vee(D_W)$ represented by the holomorphic function $-f\cdot\frac{d}{a_1}$ is sent on $(f,g)$ by $i$. On the other hand, if the point $P$ is in $\text{Supp}(Z_W)$, consider a neighborhood $V$ of $P$. The functions $-f\cdot\frac{d}{a_1}\in \sO_X(V\setminus \text{Supp}(\tilde{\eta}_1))$ and $g\cdot\frac{d}{a_2}\in \sO_X(V\setminus \text{Supp}(\tilde{\eta}_2))$ glue together on $V\setminus \text{Supp}(Z_W)$ by (\ref{relazione}). Since $Z_W$ has codimension $\geq2$, by a theorem of Hartogs they define an element $h\in\sO_X(V)$. Obviously the local section of $\mathcal{F}^\vee(D_W)$ represented by $h$ is sent on $(f,g)$.
%\end{proof}

It is easy to see by local computation that sequence (\ref{castelnuovo}) fits into the following commutative diagram
\begin{equation}
\label{diagramma}
\xymatrix {0 \ar[r] &\mathcal{F}^\vee \ar[r] \ar[d]^{\cdot D_W}&\mathcal{E}^\vee \ar[r]\ar[d]^{(-s_1,s_2)} & \sL^\vee \ar[r]\ar[d]^{\Omega}&0\\
0 \ar[r] & \mathcal{F}^\vee(D_W) \ar[r]^-{i} & \mathcal{O}_X\oplus\mathcal{O}_X\ar[r]^-{\nu} & \mathcal{F}(-D_W)\otimes \sI_{Z_W} \ar[r]&0.
}
\end{equation}
The morphism $\sE^\vee\to \mathcal{O}_X\oplus\mathcal{O}_X$ is given by contraction with the sections $-s_1$ and $s_2$, the morphism $\sL^\vee\to \mathcal{F}(-D_W)\otimes \sI_{Z_W}$ by contraction with the adjoint $\Omega$.
We can prove now the following
\begin{thm}
\label{rango1}
Let $X$ be an $m$-dimensional complex compact smooth variety. Let $\sF$, $\sL$ be invertible sheaves on $X$. Consider $\xi\in H^1(X,\mathcal{F}^\vee\otimes\sL)$ associated to the extension (\ref{sequenza}). Define $W=\left\langle \eta_1,\eta_{2}\right\rangle\subset\ker({\partial_\xi})\subset H^0(X,\mathcal{F})$ and $\Omega$ as above.
We have that $\Omega \in \Ima(H^0(X,\sL)\otimes W\to H^0(X,\det\sE))$ if and only if $\xi_{D_W}=0$.
\end{thm}

\begin{proof}
Tensoring (\ref{diagramma}) by $\sL$ and passing to cohomology we have the following diagram
\begin{equation}
\label{diagramma1}
\xymatrix {0 \ar[r] &H^0(\mathcal{F}^\vee\otimes\sL) \ar[r] \ar[d]&H^0(\mathcal{E}^\vee\otimes\sL) \ar[r]\ar[d]^-{(s_1,-s_2)} & \mathbb{C} \ar[r]\ar[d]^{\beta}& H^1(\mathcal{F}^\vee\otimes\sL)\ar[d]^\rho\\
0 \ar[r] & H^0(\mathcal{F}^\vee(D_W)\otimes\sL) \ar[r]^-{i} & H^0(\sL\oplus\sL)\ar[r]^-{\nu} & H^0(\mathcal{F}(-D_W)\otimes \sI_{Z_W}\otimes\sL) \ar[r]^-{\delta}&H^1(\mathcal{F}^\vee(D_W)\otimes\sL).
}
\end{equation}

Obviously $\beta(1)=\Omega$ and, by commutativity, $\delta(\beta(1))=\xi_{D_W}$. We have then $\xi_{D_W}=0$ if and only if $\Omega\in \Ima (H^0(\sL\oplus\sL)\stackrel{\nu}{\rightarrow} H^0(\mathcal{F}(-D_W)\otimes \sI_{Z_W}\otimes\sL))$. Since $\nu$ is given by the sections $\tilde{\eta}_2$ and $\tilde{\eta}_1$, this condition is equivalent to $\Omega \in \Ima(H^0(X,\sL)\otimes W\to H^0(X,\det\sE))$, since $\det\sE=\sF\otimes\sL$.
\end{proof}

\subsection{The Adjoint Theorem}
We go back now to the general case with $\sF$ locally free of rank $n$. By obvious identifications the natural map
\begin{equation*}
\text{Ext}^1(\det\sF,\bigwedge^{n-1}\sF\otimes\sL)\to\text{Ext}^1(\det\sF(-D_W),\bigwedge^{n-1}\sF\otimes\sL)
\end{equation*} gives an extension $\sE^{(n)}$ and a commutative diagram:

\begin{equation}
\label{diagramma2}
\xymatrix { 
&&0\ar[d]&0\ar[d]\\
0 \ar[r] & \bigwedge^{n-1}\mathcal{F}\otimes\sL \ar[r]\ar@{=}[d] &\mathcal{E}^{(n)} \ar[r]^-{\alpha} \ar[d]^-{\psi}& \det\mathcal{F}(-D_W) \ar[d]\ar[r]&0\\
0 \ar[r] &\bigwedge^{n-1}\mathcal{F}\otimes\sL \ar[r] &\bigwedge^{n}\mathcal{E} \ar[r]\ar[d] & \det\mathcal{F} \ar[r]\ar[d]&0 \\
&&\det\mathcal{F}\otimes_{\mathcal{O}_X}\mathcal{O}_{D_W}\ar@{=}[r]\ar[d]&\det\mathcal{F}\otimes_{\mathcal{O}_X}\mathcal{O}_{D_W}\ar[d]\\
&&0&0.
}
\end{equation}

%\begin{thm}[Adjoint Theorem]
%\label{teoremaaggiunta}
%Let $X$ be a $m$-dimensional complex compact smooth variety. Let $\sF$ be a rank $n$ locally free sheaf on $X$ and $\sL$ an invertible sheaf with $H^0(X,\sL)\neq0$. Consider an extension
%\begin{equation*}
%0\to\sL\to\mathcal{E}\to\mathcal{F}\to0
%\end{equation*} corresponding to $\xi\in \text{Ext}^1(\sF,\sL)\cong H^1(X,\sF^\vee\otimes\sL)$. Define $W=\left\langle \eta_1,\ldots,\eta_{n+1}\right\rangle\subset\ker({\partial_\xi})\subset H^0(X,\mathcal{F})$ and $\Omega$ as above.

%If $\Omega \in \Ima(H^0(X,\sL)\otimes \lambda^nW\to H^0(X,\det\sE))$ then $\xi\in\ker(H^1(X,\sF^\vee\otimes\sL)\to H^1(X,\sF^\vee\otimes\sL(D_W)))$.
%\end{thm}

\subsubsection{The proof of the Adjoint Theorem}

By the hypothesis $\Omega \in \Ima(H^0(X,\sL)\otimes \lambda^nW\to H^0(X,\det\sE))$ and by lemma (\ref{sollevamenti}), we can choose liftings $s_i\in H^0(X,\sE)$ of $\eta_i$ with $\Omega=0$. 

Since $D_W$ is the fixed divisor of the linear system $|\lambda^n W|$ and the sections $\omega_i$ generate this linear system, then the $\omega_i$ are in the image of
\begin{equation*}
\det\mathcal{F}(-D_W)\to\det\mathcal{F}, 
\end{equation*} so we can find sections $\tilde{\omega}_i\in H^0(X,\det\mathcal{F}(-D_W))$ such that 
\begin{equation}
\label{d}
\tilde{\omega}_i\cdot d=\omega_i,
\end{equation} where $d$ is a global section of $\sO_X(D_W)$ with $(d)=D_W$. Hence, using the commutativity of (\ref{diagramma2}), we can find liftings $\tilde{\Omega}_i\in H^0(X,\mathcal{E}^{(n)})$ of the sections $\Omega_i$. The evaluation map
\begin{equation*}
\bigoplus_{i=1}^{n+1}\mathcal{O}_X\stackrel{\tilde{\mu}}{\rightarrow}\mathcal{E}^{(n)}
\end{equation*} given by the global sections $\tilde{\Omega}_i$, composed with the map $\alpha$ of (\ref{diagramma2}), induces a map $\mu$ which fits into the following diagram
\begin{equation*}
\xymatrix { &&\bigoplus_{i=1}^{n+1}\mathcal{O}_X\ar[d]^{\tilde{\mu}} \ar@{=}[r]&\bigoplus_{i=1}^{n+1}\mathcal{O}_X\ar[d]^{\mu}\\
0 \ar[r] & \bigwedge^{n-1}\mathcal{F}\otimes\sL \ar[r] &\mathcal{E}^{(n)} \ar[r]^-\alpha & \det\mathcal{F}(-D_W) \ar[r]&0.
}
\end{equation*} We point out that the morphism $\mu$ is given by multiplication by $\tilde{\omega}_i$ on the $i$-th component. The sheaf $\Ima\tilde{\mu}$ is torsion free since it is a subsheaf of the locally free sheaf $\mathcal{E}^{(n)}$. Moreover, since $\Omega=0$, a local computation shows that $\Ima\tilde{\mu}$ has rank one outside $Z_W$.
On the other hand the sheaf $\Ima\mu$ is by definition
\begin{equation*}
\Ima\mu=\det\mathcal{F}(-D_W)\otimes \mathcal{I}_{Z_W}.
\end{equation*}

The morphism 
\begin{equation*}
\alpha\colon\mathcal{E}^{(n)}\to\det\mathcal{F}(-D_W)
\end{equation*} induces a surjective morphism, that we continue to call $\alpha$,
\begin{equation*}
\Ima\tilde{\mu}\stackrel{\alpha}{\rightarrow}\Ima\mu
\end{equation*} between two sheaves that are locally free of rank one outside $Z_W$. This morphism is also injective, because its kernel is a torsion subsheaf of the torsion free sheaf $\Ima\tilde{\mu}$, hence it is trivial.

We have proved that
\begin{equation*}
\Ima\tilde{\mu}\cong\det\mathcal{F}(-D_W)\otimes \mathcal{I}_{Z_W},
\end{equation*} so
\begin{equation*}
\sE^{(n)}\supset(\Ima \tilde{\mu})^{\vee\vee}\cong\det\mathcal{F}(-D_W) .
\end{equation*} This isomorphism gives the splitting
\begin{equation*}
\xymatrix { 0\ar[r] &\bigwedge^{n-1}\mathcal{F}\otimes\sL\ar[r]&\mathcal{E}^{(n)}\ar[r]&
\det\mathcal{F}(-D_W)\ar[r]\ar@/_1.5pc/[l]&0.
}
\end{equation*}
Since $\xi_{D_W}$ is the element of $H^1(X,\mathcal{F}^\vee\otimes\sL(D_W))$ associated to this extension, we conclude that $\xi_{D_W}=0$.

We have proved the Adjoint Theorem.

\subsubsection{An inverse of the Adjoint Theorem}

We prove now an inverse of the Adjoint Theorem.

\begin{thm}
\label{inversoaggiunta}Let $X$ be an $m$-dimensional complex compact smooth variety. 
 Let $\sF$ be a rank $n$ locally free sheaf on $X$ and $\sL$ an invertible sheaf. 
 Consider an extension $0\to\sL\to\mathcal{E}\to\mathcal{F}\to0$ corresponding to $\xi\in \text{Ext}^1(\sF,\sL)$. 
 Let $W=\langle \eta_1,\ldots,\eta_{n+1}\rangle$ be a $n+1$-dimensional sublinear system of $\ker({\partial_\xi})\subset H^0(X,\mathcal{F})$. 
 Let $\Omega\in H^0(X,\det\sE)$ be  an adjoint form associated to $W$ as above.
Assume that $H^0(X,\sL)\cong H^0(X,\sL(D_W))$. 
If $\xi\in\ker(H^1(X,\sF^\vee\otimes\sL)\to H^1(X,\sF^\vee\otimes\sL(D_W)))$, then $\Omega \in \Ima(H^0(X,\sL)\otimes \lambda^nW\to H^0(X,\det\sE))$.
\end{thm}

\begin{proof}
If $\sF$ is a rank one sheaf, then (\ref{rango1}) gives the thesis without the extra assumption $H^0(X,\sL)\cong H^0(X,\sL(D_W))$.
We assume then rank $\sF$ $\geq2$.

By (\ref{zeri}), we can write $(\Omega)=D_W+F$ with $F$ effective. In the first step of the proof we want to find a global section 
\begin{equation*}
\Omega'\in H^0(X,\bigwedge^n\mathcal{E}\otimes\sL(-F))
\end{equation*} which restricts, through the natural map
\begin{equation*}
\bigwedge^n\mathcal{E}\otimes\sL(-F)\to\det\sE(-F),
\end{equation*} to the section $d\in H^0(\det\sE(-F))$, where $(d)=D_W$.

Consider the commutative diagram:
\begin{equation*}
\xymatrix { 
&0\ar[d]&0\ar[d]&0\ar[d]\\
0\ar[r]&\bigwedge^{n-1}\mathcal{F}\otimes\sL^{\otimes2}(-F)\ar[r]\ar[d]&\bigwedge^{n}\mathcal{E}\otimes\sL(-F)\ar[r]^-{G_2}\ar[d]^-{G_1}&\det\sE(-F)\ar[r]\ar[d]&0\\
0\ar[r]&\bigwedge^{n-1}\mathcal{F}\otimes\sL^{\otimes2}\ar[r]^-{\tau}\ar[d]^-{H_1}&\bigwedge^{n}\mathcal{E}\otimes\sL\ar[r]\ar[d]&\det\sE\ar[r]\ar[d]&0\\
0\ar[r]&\bigwedge^{n-1}\mathcal{F}\otimes\sL^{\otimes2}|_F\ar[r]^-{H_2}\ar[d]&\bigwedge^{n}\mathcal{E}\otimes\sL|_F\ar[r]^-{H_3}\ar[d]&\det\sE|_F\ar[r]\ar[d]&0\\
&0&0&0.
}
\end{equation*}
By the hypothesis $\xi_{D_W}=0$ it follows easily that there exists a lifting $\tilde{\Omega}\in H^0(X,\bigwedge^{n}\mathcal{E}\otimes\sL)$ of $\Omega$. Indeed, tensor (\ref{diagramma2}) by $\sL$ and take a global lifting $f\in H^0(X, \det\sE(-D_W))$ of $\Omega$. Since $\xi_{D_W}=0$, $f$ can be lifted to a section $e\in H^0(X, \sE^{(n)}\otimes\sL)$. Define $\tilde{\Omega}:=\psi(e)$. By commutativity, $H_3(\tilde{\Omega}|_F)=0$ hence we call $\bar{\mu}\in H^0(X,\bigwedge^{n-1}\mathcal{F}\otimes\sL^{\otimes2}|_F)$ the lifting of $\tilde{\Omega}|_F$. A local computation shows that the connecting homomorphism
\begin{equation*}
\delta\colon H^0(X,\bigwedge^{n-1}\mathcal{F}\otimes\sL^{\otimes2}|_F)\to H^1(X,\bigwedge^{n-1}\mathcal{F}\otimes\sL^{\otimes2}(-F))
\end{equation*} maps $\bar{\mu}$ to $\xi_{D_W}$, which is zero by hypothesis. Then there exists a global section $$\mu \in H^0(X,\bigwedge^{n-1}\mathcal{F}\otimes\sL^{\otimes2})$$ which is a lifting of $\bar{\mu}$.
The section
\begin{equation*}
\hat{\Omega}:=\Omega-\tau(\mu)\in H^0(X,\bigwedge^{n}\mathcal{E}\otimes\sL)
\end{equation*} is a new lifting of $\Omega$ that, by construction, vanishes when restricted to $F$. We call
\begin{equation*}
\Omega'\in H^0(X,\bigwedge^{n}\mathcal{E}\otimes\sL(-F))
\end{equation*} the global section which lifts $\hat{\Omega}$. It is easy to see that
$G_2(\Omega')=d$ so $\Omega'$ is the section we wanted.

In the second part of the proof we prove that $\Omega \in \Ima(H^0(X,\sL)\otimes \lambda^nW\to H^0(X,\det\sE))$.
The global sections
\begin{equation*}
\omega_i:=\lambda^n(\eta_1\wedge\ldots\wedge\hat{\eta_i}\wedge\ldots\wedge\eta_{n+1})\in H^0(X,\det\sF)
\end{equation*} generate $\lambda^nW$ and by definition they vanish on $D_W$, that is there exist global sections $\tilde{\omega}_i\in H^0(X,\det\sF(-D_W))$ such that
\begin{equation*}
\omega_i=\tilde{\omega}_i\cdot d.
\end{equation*} We consider the commutative diagram
\begin{equation}
\label{diagramma3}
\xymatrix { 0\ar[r] &\sL(-F)\ar[r]^-\alpha\ar[d]^-{\cdot F}& W\otimes\mathcal{O}_X\ar[r]^-\gamma\ar[d]^-\beta & \bar{\mathcal{F}}\ar[r]\ar[d]^-\iota&0\\
0\ar[r]&\sL\ar[r]&\mathcal{E}\ar[r]&\mathcal{F}\ar[r]&0.
}
\end{equation} The map $\beta$ is locally defined by
\begin{equation*}
(f_1,\ldots,f_{n+1})\mapsto (-1)^nf_1\cdot s_1+\ldots+f_{n+1}\cdot s_{n+1}.
\end{equation*} 

The map $\alpha$ is defined in the following way: if $f\in\sL(-F)(U)$ is a local section, then, locally on $U$, $\alpha$ is given by
\begin{equation*}
f\mapsto (\tilde{\omega}_1(f),\cdots,\tilde{\omega}_{n+1}(f)),
\end{equation*}where we observe that the sections $\tilde{\omega}_i$ are global sections of the dual sheaf of $\sL(-F)$.
The sheaf $\bar{\sF}$ is by definition the cokernel of the first row.
Now, tensoring by $\sL^\vee$, we have 
\begin{equation}
\xymatrix { 0\ar[r] &\sO_X(-F)\ar[r]^-\alpha\ar[d]^-{\cdot F}& W\otimes\sL^\vee\ar[r]^-\gamma\ar[d]^-\beta & \bar{\mathcal{F}}\otimes\sL^\vee\ar[r]\ar[d]^-\iota&0\\
0\ar[r]&\sO_X\ar[r]&\mathcal{E}\otimes\sL^\vee\ar[r]&\mathcal{F}\otimes\sL^\vee\ar[r]&0.
}
\end{equation} 
Dualizing and tensoring again by $\sO_X(D_W)$, we obtain the commutative square

\begin{equation*}
\xymatrix {  \bigwedge^{n}W\otimes\sL(D_W)\ar[r]^-{\alpha^\vee} & \det\sE\\
\sE^\vee\otimes\sL(D_W)\ar[r]\ar[u]^{\beta^\vee}&\mathcal{O}_X(D_W)\ar[u]^-{\cdot F},
}
\end{equation*}
where we have used the isomorphism of vector spaces $W^\vee\cong\bigwedge^{n}W$, given by
\begin{equation*}
\eta^i\mapsto \eta_1\wedge\ldots\wedge\hat{\eta_i}\wedge\ldots\wedge\eta_{n+1}=:e_i
\end{equation*} where $\eta^1,\ldots,\eta^{n+1}$ is the basis of $W^\vee$ dual to the basis $\eta_1,\ldots,\eta_{n+1}$ of $W$. By definition of $\alpha$ we have that $\alpha^\vee$ is the evaluation map given by the global sections $\tilde{\omega}_i$.
Note that $\sE^\vee\otimes\sL(D_W)\cong \bigwedge^{n}\mathcal{E}\otimes\sL(-F)$.
Taking global sections we have
\begin{equation*}
\xymatrix {  \bigwedge^{n}W\otimes H^0(X,\sL(D_W))\ar[r]^-{\overline{\alpha^\vee}} & H^0(X,\det\sE)\\
H^0(X,\sE^\vee\otimes\sL(D_W))\ar[r]\ar[u]^{\overline{\beta^\vee}}&H^0(X,\mathcal{O}_X(D_W))\ar[u]^-{\cdot F}.
}
\end{equation*}
The section $\Omega'\in H^0(X,\sE^\vee\otimes\sL(D_W))$ produces in $H^0(X,\det\sE)$ the adjoint $\Omega$, so by commutativity
\begin{equation*}
\Omega=\overline{\alpha^\vee}(\overline{\beta^\vee}(\Omega')).
\end{equation*} We have
\begin{equation*}
\overline{\beta^\vee}(\Omega')=\sum_{i=1}^{n+1}c_i\cdot e_i\otimes \sigma_i,
\end{equation*} where $c_i\in \mathbb{C}$ and $\sigma_i\in H^0(X,\sL(D_W))$. By our hypothesis $H^0(X,\sL)\cong H^0(X,\sL(D_W))$, there exists sections $\tilde{\sigma_i}\in H^0(X, \sL)$ with $\sigma_i=\tilde{\sigma_i}\cdot d$. So
\begin{equation*} 
\Omega=\overline{\alpha^\vee}(\overline{\beta^\vee}(\Omega'))=\overline{\alpha^\vee}(\sum_{i=1}^{n+1}c_i\cdot e_i\otimes \sigma_i)=\sum_{i=1}^{n+1}c_i\cdot\tilde{\omega}_i\cdot \sigma_i=\sum_{i=1}^{n+1} c_i\cdot\tilde{\omega}_i\cdot d\cdot \tilde{\sigma}_i=\sum_{i=1}^{n+1} c_i\cdot\omega_i\cdot \tilde{\sigma}_i.
\end{equation*} This is exactly our thesis.
\end{proof}

By the Adjoint Theorem
%(\ref{teoremaaggiunta})
 and (\ref{inversoaggiunta}) we deduce the following
\begin{cor}
\label{aggiuntaconbasezero}
If $D_W=0$, then $\xi=0$ iff $\Omega \in \Ima(H^0(X,\sL)\otimes \lambda^nW\to H^0(X,\det\sE))$.
\end{cor}

\section{Infinitesimal Torelli Theorem for projective hypersurfaces}
In this section we want to study adjoint images in the case of smooth hypersurfaces of the projective space $\mP^n$.

\subsection{Meromorphic $1$-forms on a smooth projective hypersurface}

Let $V\subset\mP^n$ be a smooth hypersurface defined by a homogeneous polynomial $F\in \mC[\xi_0,\ldots,\xi_n]$ of degree $\deg F=d$. An infinitesimal deformation $\xi\in \text{Ext}^1(\Omega^1_V,\sO_V)$ of $V$ gives an exact sequence for the sheaf of differential forms $\Omega_V^1$:
\begin{equation}
\label{estensione}
0\to\sO_V\to \Omega^1_{\sV|V}\to\Omega^1_V\to0. 
\end{equation} We assume that $n\geq3$, hence $H^0(V,\Omega^1_V)=0$ and we can not construct the adjoint of this sequence directly, so we twist (\ref{estensione}) by a suitable integer $a$ such that $\Omega^1_V(a)$ has at least $n=\rank(\Omega^1_V)+1$ global sections. A standard computation shows that $a=2$ is enough for this purpose, so from now on we will consider the sequence
\begin{equation}
0\to\sO_V(2)\to \Omega^1_{\sV|V}(2)\to\Omega^1_V(2)\to0
\label{sequenzacon2}
\end{equation} which again corresponds to $\xi\in \text{Ext}^1(\Omega^1_V(2),\sO_V(2))\cong \text{Ext}^1(\Omega^1_V,\sO_V)\cong H^1(V,\Theta_V)$, where $\Theta_V$ denotes the sheaf of vector fields on $V$.
Denote by $\sJ$ the Jacobian ideal of $F$, that is the ideal of $\mC[\xi_0,\ldots,\xi_n]$ generated by the partial derivatives $\frac{\partial{F}}{\partial{\xi_i}}$ for $i=0,\ldots,n$.
Following \cite{Gri1}[Theorem 9.8], the deformation $\xi$ is given by a class $[R]$ of degree $d$ in the quotient $\mC[\xi_0,\ldots,\xi_n]/\sJ$. If we choose a representative $R$ of degree $d$ for this class, then $F+t R=0$, for small $t$, is the equation of the hypersurface that is the associated deformation of $V$.

Together with (\ref{sequenzacon2}), we have the conormal exact sequence
\begin{equation}
\label{sequenzaimmersione}
0\to\sO_V(-d)\to \Omega^1_{\mP^n|V}\to\Omega^1_V\to0. 
\end{equation}

If we put these sequences together we obtain the diagram
\begin{equation*}
\xymatrix{
&&&0&\\
0\ar[r]&\sO_V(2)\ar[r]&\Omega^1_{\sV|V}(2)\ar[r]&\Omega^1_V(2)\ar[u]\ar[r]&0\\
&&&\Omega^1_{\mP^n|V}(2)\ar[u]&\\
&&&\sO_V(2-d)\ar[u]&\\
&&&0\ar[u]&}
\end{equation*} which can be completed as follows
\begin{equation}
\label{diagramma4}
\xymatrix{
&&0&0&\\
0\ar[r]&\sO_V(2)\ar[r]&\Omega^1_{\sV|V}(2)\ar[u]\ar[r]&\Omega^1_V(2)\ar[u]\ar[r]&0\\
0\ar[r]&\sO_V(2)\ar[r]\ar@{=}[u]&\sG\ar[r]\ar[u]&\Omega^1_{\mP^n|V}(2)\ar[u]\ar[r]&0\\
&&\sO_V(2-d)\ar[u]\ar@{=}[r]&\sO_V(2-d)\ar[u]&\\
&&0\ar[u]&0.\ar[u]&}
\end{equation}

By \cite{Gri1} the deformation $\xi$ of (\ref{sequenzacon2}) comes from $R\in H^0(\mP^n, \sO_\mP^n(d))$, then it gives the zero element in 
$H^0(V,\Theta_{ \mP^n |V} )$, hence we have that the sheaf $\sG$ in (\ref{diagramma4}) is a direct sum $\sG=\sO_V(2)\oplus\Omega^1_{\mP^n|V}(2)$ and we have a natural morphism $\phi\colon \Omega^1_{\mP^n|V}(2)\to \Omega^1_{\sV|V}(2)$ which fits in the diagram 
\begin{equation}
\label{diagramma5}
\xymatrix{
&&0&0&\\
0\ar[r]&\sO_V(2)\ar[r]&\Omega^1_{\sV|V}(2)\ar[u]\ar[r]&\Omega^1_V(2)\ar[u]\ar[r]&0\\
0\ar[r]&\sO_V(2)\ar[r]\ar@{=}[u]&\sG\ar[r]\ar[u]&\Omega^1_{\mP^n|V}(2)\ar[u]\ar[r]\ar[ul]^-\phi&0\\
&&\sO_V(2-d)\ar[u]\ar@{=}[r]&\sO_V(2-d)\ar[u]&\\
&&0\ar[u]&0.\ar[u]&}
\end{equation}

The morphism $\phi$ gives in a natural way a morphism 
\begin{equation*}
\phi^n\colon H^0(V,\det(\Omega^1_{\mP^n|V}(2)))\cong H^0(V,\sO_V(n-1))\to H^0(V,\det(\Omega^1_{\sV|V}(2)))\cong H^0(V,\sO_V(n+d-1)).
\end{equation*}
We can write explicitly the isomorphism between $H^0(V,\det(\Omega^1_{\mP^n|V}(2)))=H^0(V,\Omega^n_{\mP^n|V}(2n))$ and $H^0(V,\sO_V(n-1))$. Note that $H^0(\mP^n,\Omega^n_{\mP^n}(2n))\to H^0(V,\Omega^n_{\mP^n|V}(2n))$ is surjective, so we will focus on the rational $n$-forms on $\mP^n$.
By \cite{Gri1}[Corollary 2.11] this forms may be written as $\omega=\frac{P\Psi}{Q}$ where $\Psi=\sum_{i=0}^n(-1)^i\xi_i(d\xi_0\wedge\ldots\wedge d\widehat{\xi_i}\wedge\ldots\wedge d\xi_n)$ gives a generator of $H^0(\mP^n,\Omega^n_{\mP^n}(n+1))$ and $\deg Q=\deg P +(n+1)$.  In our case $Q$ is a polynomial of degree $2n$, hence $P$ has degree $n-1$. This identification depends on the (noncanonical) choice of the polynomial $Q$ and gives an isomorphism $H^0(V,\Omega^n_{\mP^n|V}(2n))\to H^0(V,\sO_V(n-1))$ defined by $\omega|_V\mapsto P$.

\begin{prop}
\label{moltiplicazioneR}
$\phi^n$ is given via the multiplication by the polynomial $R$ (modulo $F$). 
\end{prop}
\begin{proof}

Locally we can see $\sV$ in the product $\Delta\times\mP^n$ of the projective space with a disk; here $\sV$ is defined by the equation $F+tR=0$. Hence $d(F+tR)=0$ in $\Omega^1_{\sV}$, that is $dF+dt\cdot R+dR\cdot t=0$. 

Call $F_i:=\frac{\partial F}{\partial \xi_i}$. Since $V$ is smooth, there exist $i$ such that $U_i=(F_i\neq0)$ is a nontrivial open subset; let for example $U_1$ be nontrivial.
Take local coordinates $z_i=\frac{\xi_i}{\xi_0}$ in the open set $(\xi_0\neq0)\cap U_1$.
Then we have 

\begin{equation}
d z_1=-\frac{Rdt}{F_1}-\frac{tdR}{F_1}-\sum_{i>1}\frac{F_i}{F_1}d z_i
\end{equation} which gives in $V$ (that is for $t=0$)
\begin{equation}
\label{dexi0}
d z_1=-\frac{Rdt}{F_1}-\sum_{i>1}\frac{F_i}{F_1}d z_i
\end{equation}
The image $\phi^n(\omega|_V)$ is then obtained by the substitution of (\ref{dexi0}) in $\frac{P(z)}{Q(z)}dz_1\wedge\ldots\wedge dz_n$, which is the local form of $\frac{P(\xi)\Psi}{Q(\xi)}$. Hence
\begin{equation}
\frac{P(z)}{Q(z)}dz_1\wedge\ldots\wedge dz_n=-\frac{P(z) R(z)}{Q(z) F_1(z)}dt\wedge dz_2\wedge\ldots\wedge dz_n.
\end{equation}
If we homogenize we obtain on $U_1$
\begin{equation*}
\frac{P\Psi}{Q}=-\frac{PR}{QF_1}\sum_{i\neq1}(-1)^{i-1}\text{sgn}(i-1)\xi_idt\wedge d\xi_0\wedge\widehat{d\xi_1}\ldots\wedge \widehat{d\xi_i}\wedge\ldots\wedge d\xi_n%\frac{P}{Q}[(\xi_0d\xi_1\wedge\ldots\wedge d\xi_n+\sum \xi_i \frac{F_i}{F_0}d\xi_1\wedge\ldots\wedge d\xi_n)+(\sum_{i=1}^n(-1)^{i-1}\xi_iRdt\wedge d\xi_1\wedge\ldots\wedge \widehat{d\xi_i}\wedge\ldots\wedge d\xi_n)].
\end{equation*} Hence
\begin{equation}
\phi^n(\omega|_V)=-\frac{PR}{QF_1}\sum_{i\neq1}(-1)^{i-1}\text{sgn}(i-1)\xi_idt\wedge d\xi_0\wedge\widehat{d\xi_1}\ldots\wedge \widehat{d\xi_i}\wedge\ldots\wedge d\xi_n
\end{equation} and it is clear that $\phi^n$ is given by multiplication with $R$.
\end{proof}   
\subsection{A canonical choice of adjoints on a hypersurface of degree $d>2$}

\medskip

We want now to construct adjoint forms associated to the sequence (\ref{sequenzacon2}).

Assume that $n\geq3$, so that $H^1(V,\sO_V(2))=H^1(V,\sO_V(2-d))=0$, and we can lift all the global sections of $H^0(V,\Omega^1_V(2))$ both in the horizontal and in the vertical sequence of (\ref{diagramma5}).

We take $\eta_1,\ldots,\eta_n\in H^0(V,\Omega^1_V(2))$ global forms and we want to find liftings $s_1,\ldots,s_n\in H^0(V,\Omega^1_{\sV|V})$. This can be done since $H^1(V,\sO_V(2))$ is zero. A generalized adjoint is then the global section of the sheaf $\det (\Omega^1_{\sV|V}(2))=\sO_V(n+d-1)$ given by $\Omega:=\Lambda^n(s_1\wedge\ldots\wedge s_n)\in H^0(V,\det (\Omega^1_{\sV|V}(2)))$. 

We point out another interesting way to compute this generalized adjoint form using Proposition (\ref{moltiplicazioneR}).

Consider the sequence (\ref{sequenzaimmersione}), that is the vertical sequence in (\ref{diagramma5}). Since $H^1(V,\sO_V(2-d))=0$, we can find liftings $\tilde{s_1},\ldots,\tilde{s_n}\in H^0(V, \Omega^1_{\mP^n|V}(2))$ of the sections $\eta_1,\ldots,\eta_n$. Furthermore they are unique if $d>2$. We can thus consider the adjoint form associated to (\ref{sequenzaimmersione}) given by $\widetilde{\Omega}:=\Lambda^n (\tilde{s_1}\wedge\ldots\wedge\tilde{s_n})$. This adjoint is independent from the deformation $\xi$; it depends only on $V$ and its embedding in $\mP^n$. If $d>2$, then $\widetilde{\Omega}$ is unique.

To describe $\widetilde{\Omega}$ explicitly we first consider the exact sequence
\begin{equation}
0\to\Omega^1_{\mP^n}(2-d)\to\Omega^1_{\mP^n}(2)\to\Omega^1_{\mP^n|V}(2)\to0.
\end{equation} If $d>2$, the vanishing of $H^0(V,\Omega^1_{\mP^n}(2-d))$ and $H^1(V,\Omega^1_{\mP^n}(2-d))$ (c.f. Bott Formulas), gives the isomorphism $H^0(V,\Omega^1_{\mP^n}(2))=H^0(V,\Omega^1_{\mP^n|V}(2))$. Hence, the forms $\tilde{s_i}$ are the restriction on $V$ of global rational $1$-forms. By \cite{Gri1}[Theorem 2.9] we can write
\begin{equation}
\label{1formerazionali}
\tilde{s_i}=\frac{1}{Q}\sum_{j=0}^n L^i_j d\xi_j
\end{equation} where $\deg Q=2$ and $L^i_j$ is a homogeneous polynomial of degree $1$ which does not contain $\xi_j$ in its expression. Hence
\begin{equation}
\widetilde{\Omega}= \Lambda^n(\tilde{s_1}\wedge\ldots\wedge\tilde{s_n})=\frac{1}{Q^n}\sum_{i=0}^n M_i d\xi_0\wedge\ldots\wedge \widehat{d\xi_i}\wedge\ldots\wedge d\xi_n
\end{equation} where $M_i$ is the determinant of the matrix obtained by 
\begin{equation}
\label{matrice}
\begin{pmatrix}
	L^1_0&\ldots&L^n_0\\
	\vdots&&\vdots\\
	L^1_n&\ldots&L^n_n
\end{pmatrix}
\end{equation} removing the $i$-th row. Since $\widetilde{\Omega}$ is a rational $n$-form on $\mP^n$, following \cite{Gri1}[Corollary 2.11] it can be written as $\frac{P\Psi}{Q^n}$, and we deduce that
\begin{equation}
\frac{M_i}{(-1)^i\xi_i}=P
\end{equation} for all $i=0,\ldots,n$. $P$ is a polynomial of degree $n-1$ and it corresponds to $\widetilde{\Omega}$ via the isomorphism $H^0(V,\Omega^n_{\mP^n|V}(2n))\cong H^0(V,\sO_V(n-1))$. Hence by (\ref{moltiplicazioneR}) we have that the form $\Omega\in H^0(V,\sO_V(n+d-1))$ given by $PR$ is a canonical choice of adjoint form for $W=\langle \eta_1,\ldots, \eta_n\rangle$ and $\xi$.
\begin{rmk}
Alternatively this can be seen using the Euler sequence on $V$:
\begin{equation}
\label{eulero}
0\to \sO_V\to\bigoplus^{n+1}\sO_V(1)\to \Theta_{\mP^n|V}\to0.
\end{equation}
This sequence, dualized and conveniently tensorized gives
\begin{equation}
0\to \Omega_{\mP^n|V}^1(2)\to \bigoplus_{i=1}^{n+1}\sO_V(1)\to \sO_{V}(2)\to 0.
\label{euleroduale2}
\end{equation} The sections $\tilde{s}_i$ are associated via the first morphism to an $n+1$-uple of linear polynomials $(L_i^0,\ldots,L_i^n)$. Then, taking the wedge product of (\ref{euleroduale2})
we obtain an exact sequence
\begin{equation}
0\to\Omega^n_{\mP^n|V}(2n)\cong\sO_V(n-1)\to \bigwedge^n\sO_V(1)=\bigoplus^{n+1}\sO_V(n)\to \Omega^{n-1}_{\mP^n|V}(2n)\to 0
\end{equation} where the morphism $\sO_V(n-1)\to\bigoplus^{n+1}\sO_V(n)$ is given by
\begin{equation}
G\mapsto (G\xi_0,\ldots,(-1)^nG\xi_n).
\end{equation} Since $\widetilde{\Omega}=\Lambda^n (\tilde{s_1}\wedge\ldots\wedge\tilde{s_n})\in H^0(V,\Omega^n_{\mP^n|V}(2n))$ is sent exactly to $(L_0^0,\ldots,L_0^n)\wedge\ldots\wedge(L_n^0,\ldots,L_n^n)=(M_0,\ldots,M_n)$ (using the same notation as above), then we conclude that $\widetilde{\Omega}$ corresponds in $H^0(V,\sO_V(n-1))$ to a polynomial $P$ which satisfies 
\begin{equation}
\frac{M_i}{(-1)^i\xi_i}=P.
\end{equation}
\end{rmk}

\subsection{The adjoint sublinear systems obtained by meromorphic $1$-forms}

To study the conditions given in (\ref{aggiuntazero1}) and (\ref{aggiuntazero2}), we need to describe the sections
\begin{equation*}
\widetilde{\Omega_i}:=\Lambda^{n-1}(\tilde{s_1}\wedge\ldots\wedge\hat{\tilde{s_i}}\wedge\ldots\wedge \tilde{s_n})\in H^0(V,\Omega^{n-1}_{\mP^n|V}(2n-2))
\end{equation*} (c.f. (\ref{Omegai})) and their images in $H^0(V,\Omega^{n-1}_{V}(2(n-1)))=H^0(V,\sO_V(n+d-3))$ that we have denoted by $\omega_i$.

A computation similar to the above shows that
\begin{equation}
\label{scrittura1}
\widetilde{\Omega_i}= \Lambda^{n-1}(\tilde{s_1}\wedge\ldots\wedge\hat{\tilde{s_i}}\wedge\ldots\wedge \tilde{s_n})=\frac{1}{Q^{n-1}}\sum_{j<k}M^i_{jk}d\xi_0\wedge\ldots\wedge\hat{d\xi_j}\wedge\ldots\wedge\hat{d\xi_k}\wedge\ldots\wedge d\xi_n
\end{equation} where $M^i_{jk}$ is the determinant of the matrix obtained by (\ref{matrice}) removing the $i$-th column and the $j$-th and $k$-th rows.
On the other hand, rearranging the expression of \cite{Gri1}[Theorem 2.9] we can write 
\begin{equation}
\label{scrittura2}
\widetilde{\Omega_i}=\frac{1}{Q^{n-1}}\sum_{j} A^i_j(\sum_{k\neq j} (-1)^{k+j}\text{sgn}(k-j)\xi_k d\xi_0\wedge\ldots\wedge\hat{d\xi_j}\wedge\ldots\wedge\hat{d\xi_k}\wedge\ldots\wedge d\xi_n)
\end{equation} with $\deg A^i_j=n-2$.

Comparing (\ref{scrittura1}) and (\ref{scrittura2}) gives
\begin{equation}
M^i_{jk}=(-1)^{j+k}(A^i_j\xi_k-\xi_j A_k^i).
\end{equation} As before this can be computed also via the Euler sequence.

We call $\Xi_j:=\sum_{k\neq j} (-1)^{k+j}\text{sgn}(k-j)\xi_k d\xi_0\wedge\ldots\wedge\hat{d\xi_j}\wedge\ldots\wedge\hat{d\xi_k}\wedge\ldots\wedge d\xi_n$. Note that the sections $\Xi_j$, for $j=0,\ldots,n$ give a basis of $H^0(V,\Omega^{n-1}_{\mP^n|V}(n))$.

\begin{prop}
$\omega_i=\sum_{j} A^i_j\cdot F_j$ in $H^0(V,\sO_V(n+d-3))$
\end{prop}
\begin{proof}
It is enough to show that the image of $\Xi_j$ through the morphism $\Omega^{n-1}_{\mP^n|V}(n)\to \sO_V(d-1)$ is $F_j$.
Consider the exact sequence of the tangent sheaf of $V$:
\begin{equation}
0\to\Theta_V\to\Theta_{\mP^n|V}\to \sO_V(d)\to 0.
\end{equation} The beginning of the Koszul complex is
\begin{equation}
\bigwedge^n\Theta_{\mP^n|V}\otimes\sO_V(-d)\to\bigwedge^{n-1}\Theta_{\mP^n|V}
\end{equation} which, tensored by $\sO_V(-n)$, gives
\begin{equation}
\label{koszul}
\bigwedge^n\Theta_{\mP^n|V}\otimes\sO_V(-n-d)\to\bigwedge^{n-1}\Theta_{\mP^n|V}\otimes\sO_V(-n).
\end{equation} This is exactly the dual of $\Omega^{n-1}_{\mP^n|V}(n)\to \sO_V(d-1)$. Hence we only need to show that the morphism (\ref{koszul}) composed with the contraction by $\Xi_i$
\begin{equation}
\bigwedge^{n-1}\Theta_{\mP^n|V}\otimes\sO_V(-n)\stackrel{\Xi_i}{\rightarrow}\sO_V
\end{equation} is the multiplication by $F_i$. This is easy to see by a standard local computation.
\end{proof}
\begin{rmk}
We immediately have that the polynomials associated to the sections $\omega_i$ are in the Jacobian ideal of $V$.
\end{rmk}
The condition (\ref{aggiuntazero2}), that is 
\begin{equation}
\Omega \in \Ima(H^0(V,\sO_V(2))\otimes \lambda^nW \to H^0(V,\sO_V(n+d-1))),
\end{equation} can be written, modulo $F$, as
\begin{equation}
RP=\sum \omega_i\cdot S_i=\sum_{i,j} A^i_j\cdot F_j\cdot S_i,
\end{equation} where $\deg S_i=2$. In particular this implies that $RP$ is in the Jacobian ideal of $V$.

\begin{prop}
The base locus $D_W$ of the linear system $|\lambda^nW|$ is zero for the generic $W$.
\end{prop}
\begin{proof}
By \cite{PZ}[Proposition 3.1.6] it is enough to prove that $H^0(V,\Omega^1_V(2))$ generically generates the sheaf $\Omega^1_V(2)$ and that $D_{H^0(V,\Omega^1_V(2))}=0$. We have an explicit basis for $H^0(V,\Omega^1_V(2))$ given by
\begin{equation}
\frac{\xi_id\xi_j-\xi_jd\xi_i}{Q}
\label{baseforme}
\end{equation} where $i<j$ and $\deg Q=2$.
The vector space $\lambda^nH^0(V,\Omega^1_V(2))\subset H^0(V,\sO_V(n+d-3))$ is obviously nonzero, hence $H^0(V,\Omega^1_V(2))$ generically generates the sheaf $\Omega^1_V(2)$.

It remains to prove that $D_{H^0(V,\Omega^1_V(2))}=0$. An easy computation (for example by induction) shows that $\lambda^nH^0(V,\Omega^1_V(2))$ contains all the polynomials of the form
\begin{equation}
\xi_{i_{1}}\xi_{i_{2}}\ldots\xi_{i_{n-2}} \frac{\partial F}{\partial{\xi_j}}
\end{equation} where $\{i_1,\ldots,i_{n-2}\}\subset \{1,\ldots,n+1\}$ and $j\notin\{i_1,\ldots,i_{n-2}\}$. Since $V$ is smooth, these polynomials do not vanish simultaneously on a divisor, hence $D_{H^0(V,\Omega^1_V(2))}=0$, and we are done.
\end{proof}

\subsection{On Griffiths's proof of infinitesimal Torelli Theorem}
In this section we will prove Theorem [C] of the Introduction.

It is well known by \cite{Gri1} that the deformation $\xi$ is trivial if and only if $R$ lies in the Jacobian ideal $\sJ$ of the variety $V$.
The following lemma gives a translation of this condition in the setting of adjoint forms.
\begin{lem}
\label{jacobiano}
$R$ is in the Jacobian ideal $\sJ$ if and only if $\Omega \in \Ima(H^0(X,\sO_V(2))\otimes \lambda^nW \to H^0(X,\sO_V(n+d-1)))$ for the generic $\Omega$.
\end{lem}

\begin{proof}
If $\Omega \in \Ima(H^0(X,\sO_V(2))\otimes \lambda^nW \to H^0(X,\sO_V(n+d-1)))$, then by the Adjoint Theorem, $\xi_{D_W}=0$. Since $D_W=0$, the deformation is trivial, hence $R$ lies in the Jacobian ideal.

Viceversa if $R\in \sJ$, the deformation is trivial and by theorem (\ref{inversoaggiunta}), we have that $\Omega \in \Ima(H^0(X,\sO_V(2))\otimes \lambda^nW \to H^0(X,\sO_V(n+d-1)))$
\end{proof}

Our theory gives another  characterization for $[R]\in( \mC[\xi_0,\ldots,\xi_n]/\sJ)_{d}\simeq H^1(X,\Theta_X)$ to be trivial.
\begin{prop}\label{dachiamare}
Assume that $\deg R=d>3$. Then $R$ is in the Jacobian ideal $\sJ$ if and only if $RP\in \sJ$ for every polynomial $P\in H^0(V,\sO_V(n-1))$ corresponding to a generalized adjoint $\widetilde{\Omega}\in H^0(V,\Omega^n_{\mP^n|V}(2n))$. 
\end{prop}
\begin{proof}
One implication is trivial.

To prove the other one the idea is to show that every monomial of $H^0(V,\sO_V(n-1))$ corresponds to a suitable generalized adjoint. Hence, if $RP\in \sJ$ for every polynomial $P\in H^0(V,\sO_V(n-1))$ corresponding to a generalized adjoint, we have that $R\cdot H^0(V,\sO_V(n-1))\subset \sJ$ and we are done by Macaulay Theorem (c.f. \cite{Vo2} Theorem 6.19 and Corollary 6.20).

We work by induction at the level of $\mP^n$, since $H^0(\mP^n,\sO_{\mP^n}(n-1))\to H^0(V,\sO_V(n-1))$ is surjective.
The base of the induction is for $n=2$. A simple computation shows that the map
\begin{equation}
\bigwedge^2H^0(\mP^2,\Omega^1_{\mP^2}(2))\to H^0(\mP^2,\sO_{\mP^2}(1))
\end{equation} is surjective because its image contains the canonical basis of degree one monomials.

For the general case we show that every monomial of degree $n-1$ is given by a generalized adjoint. Consider the natural homomorphism:
\begin{equation}
\bigwedge^nH^0(\mP^n,\Omega^1_{\mP^n}(2))\to H^0(\mP^n,\sO_{\mP^n}(n-1))
\end{equation} and take a monomial $M$ with $\deg M=n-1$. There is a variable $\xi_i$ which does not appear in $M$. We restrict to the hyperplane $\xi_i=0$ and we use induction on $\frac{M}{\xi_j}$, where $\xi_j$ appears in $M$. There exist $s_1,\ldots,s_{n-1}\in H^0(\mP^n,\Omega^1_{\mP^{n-1}}(2))$ with $s_1\wedge\ldots\wedge s_{n-1}$
which corresponds to $\frac{M}{\xi_j}$, that is 
\begin{equation}
s_1\wedge\ldots\wedge s_{n-1}=\frac{M\Psi'}{\xi_j\cdot Q^{n-1}}
\end{equation} where $\Psi'=\sum_{k=0,k\neq i}^n(-1)^k\xi_k(d\xi_0\wedge\ldots\wedge\hat{d\xi_i}\ldots\wedge\hat{d\xi_k}\ldots\wedge d\xi_n)$ gives a basis of $H^0(\mP^{n-1},\Omega_{\mP^{n-1}}^{n-1}(n))$. It is easy to see that 
\begin{equation}
s_1\wedge\ldots\wedge s_{n-1}\wedge\frac{(\xi_jd\xi_i-\xi_id\xi_j)}{Q}=\frac{M\Psi}{Q^{n}},
\end{equation} i.e. $M$ corresponds to a generalized adjoint, which is exactly our thesis.
\end{proof}

From the previous results we deduce immediately Theorem [C] of the Introduction.

\end{document}